\documentclass[12pt]{amsart}
\usepackage{amsmath,amssymb,amsbsy,amsfonts,amsthm,latexsym,mathabx,
            amsopn,amstext,amsxtra,euscript,amscd,stmaryrd,mathrsfs,
            cite,array,mathtools,enumerate}
\usepackage{url}
\usepackage[colorlinks,linkcolor=blue,anchorcolor=blue,citecolor=blue,backref=page]{hyperref}
\usepackage{color}
\usepackage{float}

\hypersetup{breaklinks=true}

\begin{document}

\newtheorem{theorem}{Theorem}
\newtheorem{lemma}[theorem]{Lemma}
\newtheorem{claim}[theorem]{Claim}
\newtheorem{cor}[theorem]{Corollary}
\newtheorem{proposition}[theorem]{Proposition}
\newtheorem{definition}[theorem]{Definition}
\newtheorem{question}[theorem]{Question}
\newtheorem{remark}[theorem]{Remark}
\newcommand{\hh}{{{\mathrm h}}}

\numberwithin{equation}{section}
\numberwithin{theorem}{section}
\numberwithin{table}{section}

\def\sssum{\mathop{\sum\!\sum\!\sum}}
\def\ssum{\mathop{\sum\ldots \sum}}
\def\dsum{\mathop{\sum \sum}}
\def\iint{\mathop{\int\ldots \int}}

\renewcommand*{\backref}[1]{}
\renewcommand*{\backrefalt}[4]{%
    \ifcase #1 (Not cited.)%
    \or        (p.\,#2)%
    \else      (pp.\,#2)%
    \fi}
\def\lc{{\mathrm{lc}}}
\newcommand{\hcan}{{\operatorname{\wh h}}}

\def\squareforqed{\hbox{\rlap{$\sqcap$}$\sqcup$}}
\def\qed{\ifmmode\squareforqed\else{\unskip\nobreak\hfil
\penalty50\hskip1em\null\nobreak\hfil\squareforqed
\parfillskip=0pt\finalhyphendemerits=0\endgraf}\fi}

\newfont{\teneufm}{eufm10}
\newfont{\seveneufm}{eufm7}
\newfont{\fiveeufm}{eufm5}
%
%
\newfam\eufmfam
     \textfont\eufmfam=\teneufm
\scriptfont\eufmfam=\seveneufm
     \scriptscriptfont\eufmfam=\fiveeufm
%
%
\def\frak#1{{\fam\eufmfam\relax#1}}

\newcommand{\bflambda}{{\boldsymbol{\lambda}}}
\newcommand{\bfmu}{{\boldsymbol{\mu}}}
\newcommand{\bfxi}{{\boldsymbol{\xi}}}
\newcommand{\bfrho}{{\boldsymbol{\rho}}}

\def\fK{\mathfrak K}
\def\fT{\mathfrak{T}}
\def\fR{\mathfrak{R}}
\def\fQ{\mathfrak{Q}}

\def\fA{{\mathfrak A}}
\def\fB{{\mathfrak B}}
\def\fC{{\mathfrak C}}

\def\E{\mathsf {E}}

\def \balpha{\bm{\alpha}}
\def \bbeta{\bm{\beta}}
\def \bgamma{\bm{\gamma}}
\def \blambda{\bm{\lambda}}
\def \bchi{\bm{\chi}}
\def \bphi{\bm{\varphi}}
\def \bpsi{\bm{\psi}}

\def\eqref#1{(\ref{#1})}

\def\vec#1{\mathbf{#1}}


\def\cA{{\mathcal A}}
\def\cB{{\mathcal B}}
\def\cC{{\mathcal C}}
\def\cD{{\mathcal D}}
\def\cE{{\mathcal E}}
\def\cF{{\mathcal F}}
\def\cG{{\mathcal G}}
\def\cH{{\mathcal H}}
\def\cI{{\mathcal I}}
\def\cJ{{\mathcal J}}
\def\cK{{\mathcal K}}
\def\cL{{\mathcal L}}
\def\cM{{\mathcal M}}
\def\cN{{\mathcal N}}
\def\cO{{\mathcal O}}
\def\cP{{\mathcal P}}
\def\cQ{{\mathcal Q}}
\def\cR{{\mathcal R}}
\def\cS{{\mathcal S}}
\def\cT{{\mathcal T}}
\def\cU{{\mathcal U}}
\def\cV{{\mathcal V}}
\def\cW{{\mathcal W}}
\def\cX{{\mathcal X}}
\def\cY{{\mathcal Y}}
\def\cZ{{\mathcal Z}}
\newcommand{\rmod}[1]{\: \mbox{mod} \: #1}

\def\cg{{\mathcal g}}

\def\e{{\mathbf{\,e}}}
\def\ep{{\mathbf{\,e}}_p}
\def\eq{{\mathbf{\,e}}_q}

\def\em{{\mathbf{\,e}}_m}

\def\Tr{{\mathrm{Tr}}}
\def\Nm{{\mathrm{Nm}}}

\def\rE{{\mathrm{E}}}
\def\rT{{\mathrm{T}}}

 \def\SS{{\mathbf{S}}}

\def\lcm{{\mathrm{lcm}}}

\def\t{\tilde}
\def\ov{\overline}

\def\({\left(}
\def\){\right)}
\def\l|{\left|}
\def\r|{\right|}
\def\fl#1{\left\lfloor#1\right\rfloor}
\def\rf#1{\left\lceil#1\right\rceil}
\def\flq#1{\langle #1 \rangle_q}

\def\mand{\qquad \mbox{and} \qquad}

\newcommand{\commIg}[1]{\marginpar{%
\begin{color}{magenta}
\vskip-\baselineskip 
\raggedright\footnotesize
\itshape\hrule \smallskip Ig: #1\par\smallskip\hrule\end{color}}}

\newcommand{\commIl}[1]{\marginpar{%
\begin{color}{red}
\vskip-\baselineskip 
\raggedright\footnotesize
\itshape\hrule \smallskip Il: #1\par\smallskip\hrule\end{color}}}

\newcommand{\commSi}[1]{\marginpar{%
\begin{color}{blue}
\vskip-\baselineskip 
\raggedright\footnotesize
\itshape\hrule \smallskip Si: #1\par\smallskip\hrule\end{color}}}




\hyphenation{re-pub-lished}

\mathsurround=1pt

\def\bfdefault{b}
\overfullrule=5pt

\def \F{{\mathbb F}}
\def \K{{\mathbb K}}
\def \Z{{\mathbb Z}}
\def \Q{{\mathbb Q}}
\def \R{{\mathbb R}}
\def \C{{\\mathbb C}}
\def\Fp{\F_p}
\def \fp{\Fp^*}

\def\Smn{S_{k,\ell,q}(m,n)}

\def\Kmn{\cK_p(m,n)}
\def\psmn{\psi_p(m,n)}

\def\SM{\cS_{k,\ell,q}(\cM)}
\def\SMN{\cS_{k,\ell,q}(\cM,\cN)}
\def\SAMN{\cS_{k,\ell,q}(\cA;\cM,\cN)}
\def\SABMN{\cS_{k,\ell,q}(\cA,\cB;\cM,\cN)}

\def\SIJq{\cS_{k,\ell,q}(\cI,\cJ)}
\def\SAJq{\cS_{k,\ell,q}(\cA;\cJ)}
\def\SABJq{\cS_{k,\ell,q}(\cA, \cB;\cJ)}

\def\sM{\cS_{k,q}^*(\cM)}
\def\sMN{\cS_{k,q}^*(\cM,\cN)}
\def\sAMN{\cS_{k,q}^*(\cA;\cM,\cN)}
\def\sABMN{\cS_{k,q}^*(\cA,\cB;\cM,\cN)}

\def\sIJq{\cS_{k,q}^*(\cI,\cJ)}
\def\sAJq{\cS_{k,q}^*(\cA;\cJ)}
\def\sABJq{\cS_{k,q}^*(\cA, \cB;\cJ)}
\def\sABJp{\cS_{k,p}^*(\cA, \cB;\cJ)}

 \def \xbar{\overline x}

\title[Exponential Sums with Trinomials]{
Multiplicative Energy of Shifted Subgroups and 
Bounds On Exponential Sums with Trinomials in Finite Fields}

 \author[S.  Macourt] {Simon Macourt}
\address{Department of Pure Mathematics, University of New South Wales,
Sydney, NSW 2052, Australia}
\email{s.macourt@unsw.edu.au}

 \author[I. D. Shkredov]{Ilya D. Shkredov}
\address{Steklov Mathematical Institute of Russian Academy
of Sciences, ul. Gubkina 8, Moscow, Russia, 119991, and Institute for Information Transmission Problems  of Russian Academy
of Sciences, Bolshoy Karet\-ny Per. 19, Moscow, Russia, 127994,
and MIPT, Institutskii per. 9, Dolgoprudnii, Russia, 141701}
\email{ilya.shkredov@gmail.com}
 
  \author[I. E. Shparlinski] {Igor E. Shparlinski}
\address{Department of Pure Mathematics, University of New South Wales,
Sydney, NSW 2052, Australia}
\email{igor.shparlinski@unsw.edu.au}

\begin{abstract} We give a new bound on collinear triples in subgroups of prime finite 
fields and use it to give some new bounds on exponential sums with  trinomials.
 \end{abstract}
\keywords{exponential sum, sparse polynomial, trinomial}
\subjclass[2010]{11L07, 11T23}

\maketitle

\section{Introduction}

\subsection{Set up}

For a prime $p$, we use $\F_p$ to denote  the finite field of $p$ elements.

For a $t$-sparse polynomial  
$$
\Psi(X) = \sum_{i=1}^t a_i X^{k_i}
$$
with some   pairwise distinct positive integer exponents $k_1, \ldots, k_t$ and 
coefficients  $a_1, \ldots, a_t\in \F_p^*$,  and a multiplicative character $\chi$ of $\F_p^*$ we define  the   sums
$$
S_\chi(\Psi) = \sum_{x\in \F_p^*} \chi(x) \ep(\Psi(x)), 
$$
where $\ep(u) = \exp(2 \pi i u/p)$ and $\chi$ is an arbitrary 
multiplicative character of $\F_p^*$. 
Certainly, the most interesting and well-studied special case is when $\chi=\chi_0$ is a 
principal character. However most of our results extend to the general case 
without any loss of strength or complication of the argument, so this is how we present them.

The main challenge here is to estimate these sums better than by the Weil bound
$$
|S_\chi(\Psi)| \le \max\{k_1, \ldots, k_t\} p^{1/2}, 
$$
see~\cite[Appendix~5, Example~12]{Weil}, by taking advantage of sparsity 
and also of the arithmetic structure of the exponents $k_1, \ldots, k_t$. 

For monomials $\Psi(X) = aX^k$ (where we can always assume that $k\mid p-1$)
the first  bound of this type is due to Shparlinski~\cite{Shp1},  which has then been improved 
and extended in various directions
by Bourgain,  Glibichuk and  Konyagin~\cite{BGK}, Bourgain~\cite{Bourg1}, 
Heath-Brown and Konyagin~\cite{HBK}, Konyagin~\cite{Kon}, Shkredov~\cite{Shkr1}, 
Shteinikov~\cite{Sht}.

Akulinichev~\cite{Aku} gives several  bounds on binomials, see also~\cite{Yu}.  Cochrane, Coffelt and  Pinner, 
see~\cite{CoCoPi1,CoCoPi2,CoPi1, CoPi2, CoPi3, CoPi4} and references therein, have given a series of other bounds
on exponential sums with sparse polynomials, some of which we present below in Section~\ref{sec:old res}.

We also remark that exponential sums with sparse polynomials
and a composite denominator have  been studied in~\cite{Bourg2, Shp2}.

Here we use a slightly different approach to improve some of the previous results. 
Our approach is reliant on reducing bounds of exponential sums with sparse polynomials
to bounds of weighted multilinear exponential sums of the type considered in~\cite{PetShp}. 
However, instead of applying the results of~\cite{PetShp} directly, we first obtain a
more precise variant for triple weighted sums over multiplicative subgroups of  $\F_p^*$, 
which could be of independent interest, 
see Lemma~\ref{lem:Bound T3} below. 

This result rests  on an extension of the bound on the number of collinear triples in 
multiplicative subgroups from~\cite[Proposition~1]{Shkr2} to subgroups of any size, 
see Theorem~\ref{thm:T(G)}. In turn,  this  gives a new bound on 
the multiplicative energy of arbitrary subgroups, see Corollary~\ref{cor:EnergyShiftSubgr}, 
and has several other applications, see Section~\ref{sec:Appl}.

Although here we concentrate on the case of trinomials
 \begin{equation}
\label{eq:Trinom} 
\Psi(X) =aX^{k}+bX^\ell + cX^m, 
 \end{equation}
our method works, without any changes, for more general sums with polynomials 
of the shape 
$$
\Psi(X) = aX^k + F(X^\ell) + G(X^m)
$$
with arbitrary polynomials $F,G \in \F_p[X]$ (uniformly in the degrees of $F$ and $G$, which essentially 
means that  they can be any functions defined on $\F_p$).

One can certainly use our approach for sums with quadrinomials reducing it 
to quadrilinear sums and  using our Lemma~\ref{Bound Dx}
in an appropriate place of the argument of the proof of~\cite[Theorem~1.4]{PetShp}. Furthermore, using results of~\cite{Bourg2,BouGlib,Gar},
one can consider the case of arbitrary sparse polynomials. 

The notation $A \ll B$ is equivalent to $|A| \leq c |B|$ for some constant $c$,
which, throughout the paper may only depend on the number of monomials 
in the sparse  polynomials under considerations. 

\subsection{Previous  results}
\label{sec:old res}

We compare our results  for trinomials~\eqref{eq:Trinom} with the 
 estimates of
 Cochrane, Coffelt and  Pinner~\cite[Equation~(1.6)]{CoCoPi1}
\begin{equation} \label{eq:CCP bound1}
S_\chi(\Psi) \ll \(\frac{k\ell m}{ \max\{ k,\ell, m\}}\)^{1/4}  p^{7/8}
\end{equation}
which  is non-trivial for $\min\{k\ell,  k m, \ell m \}<p^{1/2}$,
and of Cochrane  and  Pinner~\cite[Theorem~1.1]{CoPi1}:
\begin{equation} \label{eq:CP bound}
S_\chi(\Psi) \ll  (k\ell m)^{1/9} p^{5/6}
\end{equation}
which  is non-trivial for $k\ell m<p^{3/2}$.

We also recall the bound of Cochrane, Coffelt and  Pinner~\cite[Corollary~1.1]{CoCoPi2}
\begin{equation} \label{eq:CCP bound2}
S_\chi(\Psi) \ll  D^{1/2} p^{7/8} + (k\ell m)^{1/4}p^{5/8}
\end{equation}
where $D = \gcd(k, \ell, m, p-1)$, which  is non-trivial for $k\ell m<p^{3/2}$
and $D < p^{1/4}$.

\subsection{New results}

The following quantity is one of our main objects of study.

\begin{definition}[Collinear triples]
For sets  $\cU_1, \cU_2 \subseteq\F^*_p$ and elements $\lambda_1,\lambda_2 \in  \F^*_p$ we define $\rT_{\lambda_1,\lambda_2} (\cU_1, \cU_2)$ to be the number of solutions to  
\begin{equation}
\label{eq:T_def}
\frac{u_1-\lambda_1 v_1}{u_1-\lambda_1 w_1} = \frac{u_2- \lambda_2 v_2}{u_2- \lambda_2 w_2}, \qquad u_i,v_i,w_i \in \cU_i,  \ i=1,2.
  \end{equation}
\end{definition}
We also set 
$$
 \rT(\cU) = \rT_{1,1} (\cU, \cU).
 $$

As the relation~\eqref{eq:Colin} shows,  the triples $(u_i,v_i,w_i)$, $ i=1,2$ 
satisfying~\eqref{eq:T_def} define their collinear points.
Recent results on the quantity $\rT(\cU)$ for an arbitrary set $\cU$ can be 
found in~\cite{MPR-NRS}, where, in particular, the bound 
$$
 \rT(\cU) = \frac{|\cU|^6}{p}  + O\(p^{1/2} |\cU|^{7/2}\)
$$
is given.  This bound has been generalised in~\cite{Mac} as 
\begin{equation}
\label{eq:T-Gen U}
 \rT_{\lambda_1,\lambda_2}(\cU_1, \cU_2) = \frac{|\cU_1|^3|\cU_2|^3}{p}  
 + O\(p^{1/2}|\cU_1|^{3/2}|\cU_2|^2+|\cU_1|^{3}|\cU_2|\), 
 \end{equation}
 provided that $|\cU_1| \ge |\cU_2|$. 

Note that in~\eqref{eq:T_def}, as well as in all similar expressions of this type, we 
consider only the values of the variables for which these expressions are defined 
(that is, $u_i \ne \lambda_i w_i$,  $i=1,2$ in~\eqref{eq:T_def}). 

We begin by providing a new result on the number of collinear triples in subgroups.
More generally, for  a multiplicative subgroup $\cG$ of $\F^*_p$  we define  
$\rT_{\lambda} (\cG) =  \rT_{1,\lambda} (\cG)$ which is our main object of study.

\begin{theorem} \label{thm:T(G)}
Let $\cG$   be a multiplicative subgroup of $\F^*_p$.
Then for any  $\lambda  \in  \F^*_p$, we have
$$
 \rT_{\lambda} (\cG) - \frac{|\cG|^6}{p}  \ll  \left\{
\begin{array}{ll}
p^{1/2} |\cG|^{7/2},& \text{if $ |\cG| \ge p^{2/3}$},\\
|\cG|^5 p^{-1/2} , & \text{if $p^{2/3} > |\cG| \ge p^{1/2}\log p$},\\
|\cG|^4 \log |\cH|, & \text{if $|\cG|< p^{1/2}\log p$}. 
\end{array}
\right.
$$
\end{theorem}

\begin{remark}
\label{rem:new vs old}
Theorem~\ref{thm:T(G)} is new only for subgroups of intermediate size $p^{2/3} > |\cG|>p^{1/2}$, otherwise it is 
contained in~\cite[Proposition~1]{Shkr2},  see also Lemma~\ref{lem:sigma}  below, 
or in the bound~\eqref{eq:T-Gen U} from~\cite{Mac}.
\end{remark}

\begin{remark}
\label{rem:G and H}
The method of proof of Theorem~\ref{thm:T(G)} also works without any changes 
for $\rT_{\lambda, \mu} (\cG, \cH) $ with two multiplicative subgroups, similarly to Lemma~\ref{lem:sigma}.
However, for subgroups of significantly different sizes the optimisation part becomes rather
tedious. 
\end{remark}



We use  Theorem~\ref{thm:T(G)} to obtain the following new bound on trinomial sums.

\begin{theorem}
\label{thm:Bound3}   
Let $\Psi(X)$ be a trinomial of the form~\eqref{eq:Trinom} 
with $a,b,c  \in \F_p^*$.  
Define
$$d= \gcd(k,p-1), \qquad e = \gcd(\ell,p-1), \qquad  f =  \gcd(m,p-1)
$$
and
$$
g =\frac{d}{\gcd(d,f)},\qquad h =\frac{e}{\gcd(e,f)}.
$$
Suppose 
$f \ge g\ge h$, then 
$$
S_\chi(\Psi) \ll  \left\{
\begin{array}{ll}
p^{7/8}f^{1/8}, & \text{if $h\ge \(p  \log p\)^{1/2}$},\\
p^{15/16}(f/h)^{1/8} \(\log p\)^{1/16}, & \text{if $ g \ge \(p  \log p\)^{1/2}>h$},\\
p(f/gh)^{1/8}\(\log p\)^{1/8}, & \text{if $g< \(p  \log p\)^{1/2}$}. 
\end{array}
\right.
$$
\end{theorem}


Note that the assumption $f\ge g\ge h$ of Theorem~\ref{thm:Bound3}  does not present any  
additional restriction on the class of polynomials to which it applies as
the roles of $k$, $\ell$ and $m$ 
are fully symmetric: if $h > g$,  say, one can simply interchange $g$ and $h$ in the bound. 

We observe that the  bound of Theorem~\ref{thm:Bound3}  does not directly depend on the size of the exponents  $k$, $\ell$ and $m$  but rather  on various  greatest common divisors. In particular, it is strongest for large
 $d$ and $e$ and small  greatest common divisors $f$, $\gcd(d,f)$ and $\gcd(e,f)$. 
Furthermore,  it may remain nontrivial even for polynomials of very large degrees,  while 
 the bounds~\eqref{eq:CCP bound1},  \eqref{eq:CP bound} and~\eqref{eq:CCP bound2} 
 all become trivial for trinomials of large degree.
 Thus it is easy to give various families   of parameters where  Theorem~\ref{thm:Bound3}
 improves the bounds~\eqref{eq:CCP bound1}, \eqref{eq:CP bound} and~\eqref{eq:CCP bound2}
 simultaneously. 
 For example, we assume that $f > d > e$ are relatively prime positive integers with, say,
 $$
 p^\delta < f < (de)^{1-\delta} 
 $$
  for some fixed real $\delta > 0$. 
 Then $g = d$ and $h = e$ and we also have $d< p^{1/2}$, $e < p^{1/3}$. Hence, the bound of 
 Theorem~\ref{thm:Bound3} becomes 
 $$
 S_\chi(\Psi) \ll p(f/gh)^{1/8+o(1)}  = p(f/de)^{1/8+o(1)}, 
 $$ 
 which always gives a power saving against the trivial bound. On the other hand, 
 choosing $k$, $\ell$ and $m$ as large multiples of $d$, $e$ and $f$, respectively, 
 say, with $k,m,\ell \ge p^{1/2+\delta}$, 
 we see that all bounds from Section~\ref{sec:old res}, and of course the Weil bound, 
 are trivial. 
  
 We also give further applications of Theorem~\ref{thm:T(G)} to some additive problems
 with multiplicative subgroups of $\F_p^*$ in Section~\ref{sec:Appl}.
 In particular, in  Corollary~\ref{cor:Rom} we consider a modular version of the {\it Romanoff
theorem\/} and show that  for  almost all primes $p$, any residue  class 
modulo $p$ can be represented as a sum of  a prime $\ell   < p$ and three powers of 
any fixed integer $g \ge 2$. 
We recall that the classical result of Romanoff~\cite{Rom} asserts  that for any fixed integer $g \ge 2$ a positive proportion of integers can be written 
in the form $\ell+g^k$, with some prime $\ell$ and non-negative integer $k$. By a result of Crocker~\cite{Croc}, there are infinitely 
many positive integers not of the form  $\ell+2^k+2^m$. The case of three powers of $2$ or any other base $g>2$
is widely open.

\section{Collinear Triples}
\label{sec:col trip}

\subsection{Prelimaries}

We require some previous results.  
We note that we use  Lemma~\ref{lem:Mit} only for $\cG = \cH$, however we present it and also some other results in full generality as we  believe they may find several other applications and this deserves to be known better. 

The first one is a result of Mit'kin~\cite[Theorem~2]{Mit}
extending  that of 
Heath-Brown and 
Konyagin~\cite[Lemma~5]{HBK},   see also~\cite{Kon, 
	ShkVyu} 
for further generalisations.
 
\begin{lemma}
\label{lem:Mit}
Let $\cG$ and $\cH$ be  subgroups of $\F_p^*$ and let 
 $\cM_\cG$ and $\cM_\cH $   
be two complete sets  of distinct coset representatives 
of $\cG$ and $\cH$  in $\F_p^*$.
For an arbitrary set $\varTheta \subseteq\cM_\cG \times \cM_\cH$ such that 
$$
|\varTheta| \le \min\left\{ |\cG||\cH|, \frac{p^3}{|\cG|^2|\cH|^2}\right\}
$$ 
we have
$$
\sum_{(u,v) \in \varTheta}\left |\{(x,y) \in \cG\times \cH~:~ux+vy=1\}\right| \ll (|\cG||\cH||\varTheta|^2)^{1/3}.
$$
\end{lemma}

Note that there is a natural bijection between $\cM_\cG$, $\cM_\cH$ and some subsets  of the factor groups $\F_p^*/\cG$ and $\F_p^*/\cH$. 
So, one can think of $\Theta$ as a subset of $\F_p^*/\cG \times \F_p^*/\cH$.

Clearly, the trivial bound on the sum of Lemma~\ref{lem:Mit} is
$$
\sum_{(u,v) \in \varTheta}\left |\{(x,y) \in \cG\times \cH~:~ux+vy=1\}\right| \ll 
\min\{ |\cG|, |\cH|\}|\varTheta|.
$$
Hence if, for example,  $\cG = \cH$, then  Lemma~\ref{lem:Mit} always significantly improves this bound.

Given a line 
$$
\ell_{a,b}=\{(x,y) \in \F_p^2~:~ y = ax+b\}
$$
for some
pair $(a,b) \in \F_p^2$ and  sets  $\cA, \cB\subseteq \F_p$,  we 
denote  
$$
 \iota_{\cA,\cB}\(\ell_{a,b}\) = \left |\ell_{a,b} \cap \(\cA\times \cB\)\right|.
$$ 

The following elementary identities are well-known and no doubt have appeared, 
implicitly and explicitly, in a number of works.  

\begin{lemma}
\label{lem:Mom12}
Let $\cA, \cB\subseteq \F_p$ and $\lambda,\mu \in \F_p^*$. Then
$$
\sum_{(a,b) \in \F_p^2}    \iota_{\cA,\cB}\(\ell_{a,b}\) = \sum_{(a,b) \in \F_p^2}    \iota_{\cA,\cB}\(\ell_{\lambda a,\mu b}\) =   p |\cA||\cB|
$$
and 
$$
\sum_{(a,b) \in \F_p^2}     \iota_{\cA,\cB}\(\ell_{a,b}\) \iota_{\cA,\cB}\(\ell_{\lambda a,\mu b}\)  = |\cA|^2|\cB|^2 - |\cA| |\cB|^2  + p|\cA||\cB|.
$$
\end{lemma}

\begin{proof}  The first relation is   obvious 
as for every $(x,y,a) \in \cA\times \cB \times \F_p$ there is a unique $b = y -ax$ counted in that sum.

For the second sum, we write  
\begin{align*}
\sum_{(a,b) \in \F_p^2}  &    \iota_{\cA,\cB}\(\ell_{a,b}\)  \iota_{\cA,\cB}\(\ell_{\lambda a,\mu b}\) \\
& = 
\sum_{(u,v,x,y) \in \cA\times \cB \times \cA \times \cB}
\left|\{(a,b) \in \F_p^2~:~ v =   au + b, \ y = \lambda ax+\mu b\}\right|.
\end{align*} 
We now note that the $|\cA||\cB|$ quadruples $(u,v,x,y)\in \cA\times \cB \times \cA \times \cB$ with 
$$
(u,v) = (\lambda\mu^{-1} x, \mu^{-1} y)
$$
define exactly $p$ pairs $(a,b) = (a, v-au) \in  \F_p^2$ as above. 
Furthermore, the $|\cA| |\cB|\(|\cB|-1\)$ quadruples $(u,v,x,y) \in \cA\times \cB \times \cA \times \cB$ with 
$u =\lambda\mu^{-1} x$ but $v \ne \mu^{-1} y$ do not define any pairs $(a,b)$ as above. 
The remaining 
$$
|\cA|^2|\cB|^2 -|\cA| |\cB|\(|\cB|-1\) -|\cA||\cB| =|\cA|^2|\cB|^2 - |\cA| |\cB|^2 
$$
pairs (including the one with  $u \ne \lambda\mu^{-1} x$ but $v = \mu^{-1} y$) 
define  one pair $(a,b) \in \F_p^2$ as above each, which concludes the proof. 
\end{proof} 

Using Lemma~\ref{lem:Mom12} with $\lambda = \mu =1$,  
we now immediately derive the following result:

\begin{cor}
\label{cor:BKT-Deviate}
Let $\cA\subseteq \F_p$. Then
$$
\sum_{(a,b) \in \F_p^2} \(  \iota_{\cA,\cB}\(\ell_{a,b}\)- \frac{|\cA||\cB|}{p} \right)^2 \le p|\cA| |\cB| \,.
$$
\end{cor}

We now link the number of collinear triples $  \rT_{\lambda,\mu} (\cA, \cB) $ with the 
quantities $  \iota_{\cA,\cB}\(\ell_{a,b}\)$.

\begin{lemma}
\label{lem: T-Mom3}
Let $\cA, \cB\subseteq \F_p$ and $\lambda,\mu \in \F_p^*$. Then
$$
\rT_{\lambda, \mu} (\cA, \cB) =
\sum_{(a,b) \in \F_p^2}     \iota_{\cA,\cB}\(\ell_{a,b}\)   \iota_{\cA,\cB}\(\ell_{\lambda a,\mu  b}\)^2+ O(|\cA|^2 |\cB|^2)\,. 
$$
\end{lemma}

\begin{proof} Transforming the equation~\eqref{eq:T_def} into 
$$
\frac{u_1-\lambda v_1}{u_2-\mu v_2} = \frac{u_1-\lambda w_1}{u_2-\mu w_2}, \qquad u_1,v_1,w_1\in \cA,  \ u_2,v_2,w_2\in \cB,  
$$
we introduce an error  of magnitude $ O(|\cA|^2 |\cB|^2)$
(coming from different pairs of variables which must be distinct). 
Then collecting, for every $a \in \F_p$, the solutions with
$$
\frac{u_1-\lambda v_1}{u_2-\mu v_2} = \frac{u_1-\lambda w_1}{u_2-\mu w_2} = a
$$
we derive:
$$
u_1-au_2 = \lambda v_1-a\mu v_2 = \lambda w_1-a\mu w_2.
$$
We now denote this common value by $b$ and observe that for any $(a,b) \in \F_p^2$ there are 
$  \iota_{\cA,\cB}\(\ell_{a,b}\)   \iota_{\cA,\cB}\(\ell_{\lambda a,\mu  b}\)^2$ solutions to 
\begin{equation}
\label{eq:Colin}
u_1-au_2 = \lambda v_1-a\mu v_2 = \lambda w_1-a\mu w_2= b.
  \end{equation}
Summing over all pairs $(a,b)  \in \F_p^2$, we obtain the result.
 \end{proof} 
 
\begin{cor}
\label{cor:T Asymp}
Let $\cA, \cB\subseteq \F_p$ and $\lambda,\mu \in \F_p^*$. Then
\begin{align*}
\rT_{\lambda, \mu} (\cA, \cB)  &  -  \frac{|\cA|^3|\cB|^3}{p} \\
&  = 
\sum_{(a,b) \in \F_p^2}   \iota_{\cA,\cB}\(\ell_{a,b}\)  \(  \iota_{\cA,\cB}\(\ell_{\lambda a,\mu  b}\)- \frac{|\cA||\cB|}{p} \right)^2
+O(|\cA|^2 |\cB|^2)\,. 
\end{align*}
\end{cor}

\begin{proof} 
Using the identity $X^2 = (X-Y)^2 +2XY-Y^2$ with $X =    \iota_{\cA,\cB}\(\ell_{\lambda a,\mu  b}\)$ and $Y =  |\cA| |\cB|/p$
we see that 
\begin{equation}
\begin{split}
\label{eq:TR1R2}
\sum_{(a,b) \in \F_p^2}  &   \iota_{\cA,\cB}\(\ell_{a,b}\)  \iota_{\cA,\cB}\(\ell_{\lambda a,\mu  b}\)^2\\
 & = \sum_{(a,b) \in \F_p^2}   \iota_{\cA,\cB}\(\ell_{a,b}\)  \(  \iota_{\cA,\cB}\(\ell_{\lambda a,\mu  b}\)- \frac{|\cA|^2}{p} \right)^2 +   R_1 -  R_2, 
\end{split}
\end{equation}
where
\begin{align*}
R_1 & = 2  \frac{|\cA||\cB|}{p}  \sum_{(a,b) \in \F_p^2}     \iota_{\cA,\cB}\(\ell_{a,b}\)
  \iota_{\cA,\cB}\(\ell_{\lambda a,\mu  b}\), \\
R_2 &= 
  \frac{|\cA|^2|\cB|^2}{p^2}   \sum_{(a,b) \in \F_p^2}     \iota_{\cA,\cB}\(\ell_{a,b}\).
\end{align*}
By Lemma~\ref{lem:Mom12}, after simple calculations,  we have 
$$
R_1 -  R_2 =  2  \( |\cA|^2|\cB|^2- |\cA|^2 |\cB|^3/p\)
 \ll  |\cA|^2|\cB|^2.
 $$
Combining this with~\eqref{eq:TR1R2} yields
\begin{align*}
&\sum_{(a,b) \in \F_p^2}     \iota_{\cA,\cB}\(\ell_{a,b}\)  \iota_{\cA,\cB}\(\ell_{\lambda a,\mu  b}\)^2 \\
 &  \quad =   \frac{|\cA|^3|\cB|^3}{p} + 
\sum_{(a,b) \in \F_p^2}   \iota_{\cA,\cB}\(\ell_{a,b}\)  \(  \iota_{\cA,\cB}\(\ell_{a,b}\)- \frac{|\cA|^2}{p} \right)^2+ O\( |\cA|^2|\cB|^2\). 
\end{align*}
Hence, using  Lemma~\ref{lem: T-Mom3}, we  obtain the result. 
\end{proof}

Given two sets $\cU, \cV \subseteq  \F_p$, we define $\rE^\times(\cU, \cV)$ to be the {\it multiplicative energy\/}
of $\cU$ and $\cV$, that is,    the number of solutions to 
$$
u_1v_1=u_2v_2, \qquad u_1,u_2 \in \cU, \ v_1, v_2 \in \cV.
$$
For $\cU = \cV$ we also write 
$$
\rE^\times(\cU) = \rE^\times(\cU, \cU).
$$
It is easy to see that for any subgroup of $\cG,\cH \subseteq\F_p^*$ and $\lambda,\mu \in \F_p^*$
we have 
\begin{equation}
\begin{split}
\label{eq:T and E}
\rT_{\lambda,\mu} (\cG, \cH)& = \sum_{(g,h)\in \cG\times \cH}
 \rE^\times (\cG-\lambda g,\cH-\mu h) +  O(|\cG|^3|\cH|).\\
 &= |\cG| |\cH| \rE^\times (\cG-\lambda ,\cH-\mu) +  O(|\cG|^3|\cH|),
\end{split}
\end{equation}
where the error term $O(|\cG|^3|\cH|)$ (which is obviously negative) accounts for zero values  of the linear forms in the definition of $\rT_{\lambda,\mu} (\cG, \cH)$. 

Finally, we need the following bound for small subgroups, which is a slightly simplified form of~\cite[Proposition~1]{Shkr2} combined with~\eqref{eq:T and E}.

\begin{lemma}
\label{lem:sigma}
Let $\cG$ be a subgroup of $\F_p^*$
with  $|\cG|\ge |\cH|$ and $|\cG||\cH| < p$. 
Then
$$
\rT_{\lambda,\mu} (\cG, \cH) \ll   |\cG|^3  |\cH| \log |\cG|  \,.
$$
\end{lemma}

\subsection{Initial reductions}

The argument below follows~\cite{Shkr2, Shkr4}. 

First of all, note that  Lemma~\ref{lem:sigma} implies the required result provided 
$|\cG||\cH| < p$ while the bound~\eqref{eq:T-Gen U}  implies it for $|\cG| \ge p^{2/3}$.

So it remains to consider the case 
$$
p^{2/3} > |\cG| > p^{1/2}.
$$

Let $\Delta\ge 3$ be a parameter to be chosen  later.
Using Corollaries~\ref{cor:BKT-Deviate} and~\ref{cor:T Asymp}, we obtain
\begin{equation}\label{eq:T Sigma}
\rT_\lambda(\cG) - \frac{|\cG|^6}{p}  \ll 
 |\cG|^4 + \Delta |\cG|^2 p  + W \,, 
\end{equation}
where 
$$
W =  \sum_{\substack{(a,b) \in \F_p^2\\ \iota_\cG\(\ell_{a,b}\)> \Delta}} \iota_\cG\(\ell_{a,b}\) 
 \( \iota_\cG\(\ell_{a, \lambda b}\)- \frac{|\cG|^2}{p} \right)^2.
$$

Clearly, the contribution  to $W$ from lines  with $ab= 0$, is at most $|\cG|^4$
as in this case  $\iota_\cG\(\ell_{a,b}\) = 0$ unless $a\in \cG$ or $b\in \cG$, in which case 
$\iota_\cG\(\ell_{a,b}\) = |\cG|$. 
Therefore, 
$$
  \sum_{\substack{(a,b) \in \F_p^2\\ ab = 0}}\iota_\cG\(\ell_{a,b}\)  
  \( \iota_\cG\(\ell_{a, \lambda b}\)- \frac{|\cG|^2}{p} \right)^2  =  O\(|\cG|^4\). 
$$
 Thus
\begin{equation}
\label{eq:Sigmas}
W = W^* + O\(|\cG|^4\)
\end{equation}
where
$$
W^* =  \sum_{\substack{(a,b) \in (\F_p^*)^2\\ \iota_\cG\(\ell_{a,b}\)> \Delta}} \iota_\cG\(\ell_{a,b}\)  \( \iota_\cG\(\ell_{a, \lambda b}\)- \frac{|\cG|^2}{p} \right)^2, 
$$
which is the sum we now consider.
%
%


Returning  to~\eqref{eq:T_def}, we see that the quantity $\rT_\lambda(\cG)$ is equal, up to the error
$O(|\cG|^4)$,  which can be absorbed in the same error term that  is already present
in~\eqref{eq:T Sigma}, 
to the number of  solutions of the equation 
\begin{equation}
\begin{split}
\label{eq:Tmod}
(u_1- v_1)(u_2-\lambda w_2)& = (u_1-w_1)(u_2-\lambda v_2)\ne 0,  \\ 
u_i,v_i,w_i & \in \cG,  \ i=1,2.
\end{split}
\end{equation}

\subsection{Sets $ \varTheta_\tau$ and $\cQ_\tau$}
Let, as before, $\cM_\cG$  be a set of distinct coset representatives 
of $\cG$   in $\F_p^*$.
Take another  parameter $\tau  \ge \Delta$ and put
$$
    \varTheta_\tau= \{ (\alpha, \beta) \in \cM_\cG^2  ~:~ | \{ (x,y) \in \cG^2~:~\alpha x+ \beta y=1 \} | \ge \tau \} \,.
$$
In other words, $\varTheta_\tau$  is the set of $(\alpha, \beta) \in \cM_\cG^2$ for which
the lines 
\begin{equation}\label{eq:L and l}
\cL_{\alpha, \beta} = \{ (x,y)\in \F_p^2~:~\alpha x+ \beta y=1 \} =  \ell_{-\alpha \beta^{-1}, \beta^{-1}}
\end{equation}
have the intersection with $\cG^2$  of size at least 
$$ 
\iota_\cG\(\ell_{-\alpha \beta^{-1}, \beta^{-1}}\) \ge \tau.
$$
In particular, 
\begin{equation}
\label{eq: Theta-tau}
 \varTheta_\tau =
  \{  (\alpha, \beta)  \in  \cM_\cG^2 ~:~  \iota_\cG\(\cL_{\alpha, \beta}\)  \ge \tau \}.
\end{equation}

By Lemma~\ref{lem:Mit}, we have $|\varTheta_\tau| \tau \ll (|\cG||\varTheta_\tau|)^{2/3}$
provided
\begin{equation}
\label{eq:cond1}
|\cG|^4 |\varTheta_\tau| < p^3
\end{equation}
and 
\begin{equation}
\label{eq:cond2}
|\varTheta_\tau| \le   |\cG|^2.
\end{equation}

We also define the set 
\begin{equation}
\label{eq: Q-tau}
\cQ_\tau =
  \{  (\alpha, \beta)  \in \(\F_p^*\)^2 ~:~  \iota_\cG\(\cL_{\alpha, \beta}\)  \ge \tau \}.
\end{equation}
Comparing~\eqref{eq: Theta-tau} and~\eqref{eq: Q-tau}, we see that we can think of
$\varTheta_\tau$ as of an union of cosets $\cQ_\tau/\cG$. 
Clearly, we have 
\begin{equation}\label{eq:q_tau}
|\cQ_\tau| = |\cG|^2 |\varTheta_\tau|   \ll |\cG|^4 \tau^{-3}
\end{equation}
provided the conditions~\eqref{eq:cond1} and~\eqref{eq:cond2} are satisfied.

The condition~\eqref{eq:cond2}  is trivial to verify. Indeed, since
$ |\cG|^2 > p$, we have 
$$
|\varTheta_\tau| \le   |\cM_\cG|^2 = (p-1)^2/  |\cG|^{2}  \le  |\cG|^{2} 
$$
and thus~\eqref{eq:cond2}  holds.

We now show that  the condition~\eqref{eq:cond1}  also holds
for the following choice 
\begin{equation}
\label{eq: Delta}
\Delta = c |\cG|^3 p^{-3/2},
\end{equation}
with a sufficiently large constant $c$  (recalling that  $|\cG|> p^{1/2}$ we see that the 
condition $\Delta \ge 3$ is satisfied). 

\begin{lemma}
\label{lem:Q_tau} For $\Delta$ given by~\eqref{eq: Delta} the bound~\eqref{eq:cond1}  holds.
\end{lemma}

\begin{proof} 
Suppose, to the contrary,  that 
\begin{equation}
\label{eq:Contr}
|\varTheta_\tau| > p^3 / |\cG|^4\, .
\end{equation}
Whence, 
the number of incidences between points of $\cP = \cG^2$ 
and the lines $\cL_{\alpha, \beta}$ as above  with $ (\alpha, \beta) \in \cQ_\tau$ is at least 
\begin{equation}
\label{eq:L-Bound}
|\cQ_\tau| \tau =  |\cG|^2 |\varTheta_\tau|  \tau  > p^3 |\cG|^{-2} \Delta\,.
\end{equation}
On the other hand,  by a classical result which holds over any field 
(see, for example~\cite[Corollary 5.2]{BKT} 
or~\cite[Exercise~8.2.1]{TaoVu}) the number of incidences for 
any set of points $\cP$ 
and a set of  lines $\cQ_\tau$ is at most $|\cQ_\tau|^{1/2} |\cP| + |\cQ_\tau|$.
Hence 
\begin{equation}
\label{eq:Incid}
|\cQ_\tau| \tau \le |\cQ_\tau|^{1/2} |\cP| + |\cQ_\tau|
\end{equation}
 and we obtain
\begin{equation}
\label{eq:U-Bound}
|\cQ_\tau| \tau^2 \ll  |\cP|^2  =  |\cG|^4.
\end{equation}
Combining~\eqref{eq:L-Bound} and~\eqref{eq:U-Bound}, we derive
\begin{equation}
\label{eq: Bounds}
p^3 |\cG|^{-2} \Delta < |\cQ_\tau| \tau   \ll  |\cG|^4 \tau^{-1} \le   |\cG|^4 \Delta^{-1} .
\end{equation}
Recalling that  $|\cG| \ge p^{1/2} $,  we see that for $\Delta$ given by~\eqref{eq: Delta}  
with a sufficiently large constant $c$ 
the inequalities~\eqref{eq: Bounds} are impossible, which also shows that our assumption~\eqref{eq:Contr} is false and this concludes the proof. 
\end{proof}


\subsection{Concluding the proof of Theorem~\ref{thm:T(G)}}

We now define 
$$
\cR_\tau =
 \left \{  (\alpha, \beta)  \in \(\F_p^*\)^2 ~:~ 
  \max \left\{\iota_\cG\(\cL_{\alpha, \beta}\), \iota_\cG\(\cL_{\alpha, \lambda \beta}\) \right\} \ge \tau \right\}.
 $$

By Lemma~\ref{lem:Q_tau}, for the  choice~\eqref{eq: Delta} of $\Delta$ we have the desired condition~\eqref{eq:cond1}
for any $\tau \ge \Delta$. Hence, the bound~\eqref{eq:q_tau} also implies 
that 
\begin{equation}\label{eq:r_tau}
|\cR_\tau| = |\cG|^2 |\varTheta_\tau|   \ll |\cG|^4 \tau^{-3}.
\end{equation}

We see from~\eqref{eq:L and l} that there is a one-to-one correspondence 
between the lines $\ell_{a,b}$, $(a,b)\in \(\F_p^*\)^2$ and the lines 
$\cL_{\alpha, \beta}$, $(\alpha, \beta)\in \(\F_p^*\)^2$. 
We now define
$$
\tau_j = e^{j}\Delta,  \qquad j =0,1, \ldots, J,
$$
where 
$$
J = \rf{\log (|\cG|/\Delta)}.
$$
Note that due to the choice of $\Delta$ and the condition  $|\cG| \ge p^{1/2}$ we have
$$
\tau_j \ge \tau_0 =\Delta \gg  |\cG|^3 p^{-3/2} \ge   |\cG|^2/p ,  \qquad j =0,1, \ldots, J.
$$
Then, recalling also the bound~\eqref{eq:r_tau}, we conclude that the contribution to $W^*$ from the lines with 
$\tau_{j+1}\ge \iota_\cG\(\ell_{a,b}\)> \tau_j$ is bounded
by 
\begin{equation}\label{eq:Qj}
\left|\cQ_{\tau_j}\right| \tau_{j+1} \(\tau_{j+1} + |\cG|^2/p\)^2 \ll \left|\cQ_{\tau_j}\right| \tau_{j+1}^3
 \ll  |\cG|^4 .
\end{equation}
Summing up~\eqref{eq:Qj} we obtain
$$
W^* \ll  J  |\cG|^4  \ll \ |\cG|^4 \log |\cG|.
$$
Substituting this bound in~\eqref{eq:Sigmas} and 
combining it with~\eqref{eq:T Sigma}, we 
obtain
$$
 \rT_\lambda(\cG) = \frac{|\cG|^6}{p}  + O\( |\cG|^5 p^{-1/2}+  |\cG|^4 \log |\cG|\) \,
 $$
 in the range $p^{2/3} \ge |\cG| \ge p^{1/2}$, which concludes the proof.

\begin{remark}
\label{rem:Incid}
In principle, a stronger version of the classical incidence bound which is used~\eqref{eq:Incid}
may lead to improvements of Theorem~\ref{thm:T(G)}. However, the range where such improvements 
are known is far away from the range which appears in our applications, see~\cite{StdeZe}.
\end{remark}

\section{Trinomial  sums}
\subsection{Preliminaries}
%
%

We recall the following classical  bound of  bilinear sums, 
see, for example,~\cite[Lemma~4.1]{Gar}.

\begin{lemma}
\label{lem:bilin} 
For any sets $\cX, \cY \subseteq \F_p$ and any  $\alpha= (\alpha_{x})_{x\in \cX}$, $\beta = \( \beta_{y}\)_{y \in \cY}$, 
with 
$$
\max_{x\in \cX}|\alpha_{x}| \le 1 \mand  \max_{y \in \cY}|\beta_{y}| \le 1
$$
we have 
$$
\left |\sum_{x \in \cX}\sum_{y \in \cY} \alpha_{x} \beta_{y}  \ep(xy) \right| \le \sqrt{p|X||Y|}.
$$
\end{lemma}

\begin{definition}[Ratios of differences]
For a set  $\cU \subseteq\F^*_p$, we define $D_\times(\cU)$ to be the number of solutions of 
$$
(u_1-v_1)(u_2-v_2) = (u_3-v_3)(u_4-v_4), \qquad u_i,v_i \in \cU,\  i=1,2,3,4,
$$
\end{definition}

As before we define $\rT(\cU)$  as the number of solutions to~\eqref{eq:T_def}. 

We now recall the following bound from~\cite[Lemma~2.7]{PetShp}.

\begin{lemma} \label{Bound Dx}
For any set $\cU \subseteq\F_p^*$ with $|\cU|=U$, we have
$$
D_\times(\cU) \ll U^2\rT(\cU) + U^6.
$$
\end{lemma}

Combining Lemma~\ref{Bound Dx} with Theorem~\ref{thm:T(G)} we obtain
$$
D_\times(\cG) \ll \frac{|\cG|^8}{p} +  \left\{
\begin{array}{ll}
p^{1/2} |\cG|^{11/2},& \text{if $ |\cG| \ge p^{2/3}$},\\
|\cG|^7 p^{-1/2} , & \text{if $p^{2/3} > |\cG| \ge p^{1/2}\log p$},\\
|\cG|^6 \log |\cG|, & \text{if $|\cG|< p^{1/2}\log p$}. 
\end{array}
\right.
$$

Since for $|\cG| \ge  \(p  \log p\)^{1/2}$ the first term dominates, this simplifies as
\begin{cor} \label{Bound Dx2}
For a multiplicative subgroup $\cG \subseteq \F_p^*$, we have
$$
D_\times(\cG) \ll   \left\{
\begin{array}{ll}
|\cG|^8 p^{-1} , & \text{if $|\cG| \ge  \(p  \log p\)^{1/2}$},\\
|\cG|^6 \log |\cG|, & \text{if $|\cG|<  \(p  \log p\)^{1/2}$}. 
\end{array}
\right.
$$
\end{cor}

%


Substituting in Corollary~\ref{Bound Dx2} into the proof of~\cite[Theorem~1.3]{PetShp}, we obtain the following result for trilinear sums over subgroups, which improves its general bound. 

\begin{lemma} \label{lem:Bound T3}
For any multiplicative subgroups $\cF, \cG, \cH \subseteq \F_p^*$ of cardinalities $F, G, H$, respectively, with 
$F \ge G \ge  H$
and weights 
$\rho= (\rho_{u,v})$,  $\sigma = (\sigma_{u,w})$ and $\tau=(\tau_{v,w})$
with 
$$
\max_{(u,v) \in \cF \times \cG} |\rho_{u,v}| \le 1, \quad 
\max_{(u,w) \in \cF \times \cH} |\sigma_{u,w}| \le 1, \quad 
\max_{(v,w) \in \cG \times \cH} |\tau_{v,w}| \le 1,  
$$
for the sum
$$
T = \sum_{u \in\cF} \sum_{v \in \cG}
 \sum_{w\in \cH} \rho_{u,v} \sigma_{u,w} \tau_{v,w} \ep(auvw)
$$
we have
$$
T\ll  \left\{
\begin{array}{ll}
F^{7/8}GH , & \text{if $ H \ge  \(p  \log p\)^{1/2}$},\\
p^{1/16}F^{7/8}GH^{7/8 } \(\log p\)^{1/16}, & \text{if $ G \ge  \(p  \log p\)^{1/2}>H$},\\
p^{1/8}F^{7/8}G^{7/8}H^{7/8}\(\log p\)^{1/8}, & \text{if $G<  \(p  \log p\)^{1/2}$}, 
\end{array}
\right.  
$$
uniformly over $a\in \F_p^*$.
\end{lemma}

\begin{proof}
We see from~\cite[Equation~(3.8)]{PetShp} that 
\begin{equation*}
T^8 \ll  p F^7 G^4 H^4 K + F^8 G^8 H^6, 
\end{equation*}
where $K$ is the number of solutions to the equation 
\begin{align*}
(u_1-u_2)(w_1-w_2) &= (u_3-u_4)(w_3-w_4) \ne 0, \\ \
 (u_i,w_i) \in \cG&\times \cH, \quad i =1,2,3,4.
\end{align*}
As in the proof of~\cite[Theorem~1.3]{PetShp}, expressing $K$ via multiplicative 
character sums and using the Cauchy inequality, we obtain 
 $K^2 \le D_\times(\cG)D_\times(\cH)$. Applying Corollary~\ref{Bound Dx2}, 
 instead of~\cite[Equation~3.9]{PetShp}, we now obtain 
\begin{align*}
K \ll  \left\{
\begin{array}{ll}
G^4H^4/p , & \text{if $ H \ge \(p  \log p\)^{1/2}$},\\
G^4 H^{3}p^{-1/2} (\log p)^{1/2} , & \text{if $ G \ge\(p  \log p\)^{1/2}>H$},\\
(GH)^3\log p, & \text{if $G< \(p  \log p\)^{1/2}$}. 
\end{array}
\right.
\end{align*}
We now
deal with the three cases separately. 

For $H\ge \(p  \log p\)^{1/2}$ we have
$$
T^8 \ll F^7G^8H^8 + F^8G^8H^6 .
$$
Since 
$F < p < H^2$, 
the first term dominates, and we obtain
\begin{equation}\label{eq T1}
T \ll F^{7/8}GH.
\end{equation}

For $G\ge \(p  \log p\)^{1/2} > H$, we have
$$
T^8 \ll p^{1/2}F^7G^8 H^7 (\log p)^{1/2}+ F^8G^8H^6
$$
or 
\begin{equation}
\label{eq T2 prelim}
T \ll p^{1/16}F^{7/8}GH^{7/8}\(\log p\)^{1/16}+FGH^{3/4}.
\end{equation}
The first term of \eqref{eq T2 prelim} dominates for $p^{1/2}\ge F/H$. 

We now  note that 
 by Lemma~\ref{lem:bilin} and the trivial bound for the sum over $\cH$, we also have
\begin{equation}
\label{eq T trivial}
T\ll p^{1/2}F^{1/2}G^{1/2}H. 
\end{equation}
Furthermore, since  for  $F > p^{1/2}H$ and $G > p^{1/2}$ we have
\begin{align*}
p^{1/2}F^{1/2}G^{1/2}H & =   p^{1/16}F^{7/8}GH^{7/8} \left(\frac{p^{7/2} H}{F^3G^4}\right)^{1/8} \\
&< p^{1/16}F^{7/8}GH^{7/8} \left(\frac{p^3}{F^2G^4}\right)^{1/8}  < p^{1/16}F^{7/8}GH^{7/8}, 
\end{align*}
we see that for $G\ge \(p  \log p\)^{1/2} > H$ the bound~\eqref{eq T2 prelim}
simplifies as
\begin{equation}\label{eq T2}
T \le p^{1/16}F^{7/8}GH^{7/8}\(\log p\)^{1/16}.
\end{equation}

For $G<\(p\log p\)^{1/2}$, we have 
$$
T^8 \ll pF^7G^{7}H^{7}\log p + F^8G^8H^6
$$
or 
\begin{equation} 
\label{eq T3 prelim}
T \ll p^{1/8}F^{7/8}G^{7/8}H^{7/8} \(\log p\)^{1/8}+FGH^{3/4}.
\end{equation}
The first term of~\eqref{eq T3 prelim} dominates  for $pH\ge FG$.
Otherwise,  that is, for $pH< FG$, we have
\begin{align*}
p^{1/2}F^{1/2}G^{1/2}H &= p^{1/8}F^{7/8}G^{7/8}H^{7/8} \left(\frac{p^3H}{F^3G^3}\right)^{1/8} \\
&<   p^{1/8}F^{7/8}G^{7/8}H^{7/8} \left(\frac{1}{H^2}\right)^{1/8}  \le p^{1/8}F^{7/8}G^{7/8}H^{7/8}.
\end{align*}
Thus, using~\eqref{eq T trivial} we see that  the bound~\eqref{eq T3 prelim}
simplifies as
\begin{equation}\label{eq T3}
T \le p^{1/8}F^{7/8}G^{7/8}H^{7/8} \(\log p\)^{1/8}. 
\end{equation}

Combining~\eqref{eq T1}, \eqref{eq T2} and~\eqref{eq T3}, 
we complete the proof.
\end{proof}

Clearly, the bound of Lemma~\ref{lem:Bound T3} is nontrivial when $F$, $G$ and $H$ 
are all a little larger than $p^{1/3}$. More formally, for any $\varepsilon > 0$ there exists some 
$\delta>0$ such that if 
$F \ge G \ge H \ge p^{1/3+\varepsilon}$ then  the exponential sums 
of Lemma~\ref{lem:Bound T3} are bounded by $O\(FGH p^{-\delta}\)$. 
%
%

\subsection{Proof of Theorem \ref{thm:Bound3}}

Let $\cG_d$ and $\cG_e$ be the 
subgroups of $\F_p^*$ formed by the elements of orders dividing $d$ and $e$, respectively.

We have,
\begin{align*}
S_\chi(&\Psi)  = \frac{1}{de} \sum_{y \in \cG_d} \sum_{z \in \cG_e} \sum_{x\in \F_p^*}  \chi(xyz)\ep(\Psi(xyz))\\
& =  \frac{1}{de}  \sum_{x\in \F_p^*}  \sum_{y \in \cG_d}  \sum_{z \in \cG_e} 
 \chi(x) \chi(y) \chi(z)  \ep\(ax^kz^k + bx^\ell y^\ell + cx^my^mz^m\) \\
& =  \frac{1}{de}  \sum_{x\in \F_p^*} \sum_{z \in \cG_e}  
 \sum_{y \in \cG_d} \rho_{x,y}\sigma_{x,z}   \ep\(cx^my^mz^m\),
\end{align*}
where 
\begin{align*}
 \rho_{x,y} & = \chi(x) \chi(y) \ep\(bx^\ell y^\ell\) \mand \sigma_{x,z}=  \chi(z) \ep\(ax^kz^k\).
 \end{align*}

Clearly,   the set $\cX = \{x^m~:~x \in \F_p^*\}$ of non-zero $m$th powers  contains $(p-1)/f$ elements, each appearing 
 with multiplicity $f$. Furthermore, direct examination shows that  the sets  
 $\cY= \{y^m~:~y \in \cG_d\}$ and   $\cZ=\{z^m~:~z \in~\cG_e\}$ contain $g$ and $h$
elements  with multiplicities  $\gcd(d,f)$ and $\gcd(e,f)$, respectively. 
We recall that by our assumption we have $f \ge g \ge h$ and invoke 
Lemma~\ref{lem:Bound T3}, which gives us,
\begin{align*}
S_\chi(\Psi)&\ll \frac{f\gcd(d,f)   \gcd(e,f)}{de} \times \\ &\qquad \quad  \left\{
\begin{array}{ll}
(p/f)^{7/8}gh, & \text{if $h\ge p^{1/2}\log p$},\\
p^{1/16}(p/f)^{7/8}gh^{7/8} \(\log p\)^{1/16}, & \text{if $ g \ge \(p  \log p\)^{1/2}>h$},\\
p^{1/8}(p/f)^{7/8}g^{7/8}h^{7/8}\(\log p\)^{1/8}, & \text{if $g< \(p  \log p\)^{1/2}$}, 
\end{array}
\right. \\
&=  \left\{
\begin{array}{ll}
p^{7/8}f^{1/8}, & \text{if $h\ge \(p  \log p\)^{1/2}$},\\
p^{15/16}f^{1/8}h^{-1/8} \(\log p\)^{1/16}, & \text{if $ g \ge \(p  \log p\)^{1/2}>g$},\\
pf^{1/8}g^{-1/8}h^{-1/8} \(\log p\)^{1/8}, & \text{if $g< \(p  \log p\)^{1/2}$}. 
\end{array}
\right.
\end{align*}
This concludes the proof.

\section{Further Applications}
\label{sec:Appl}

\subsection{Additive properties of subgroups}
\label{sec:add prop subgr}

As usual, given a rational function 
$$R(X_1, \ldots, X_m) \in
\F_p(X_1, \ldots, X_m), 
$$ 
and $m$ sets $\cA_1, \ldots, \cA_m
\subseteq \F_p$, we define the set
 \begin{align*}
R(\cA_1,& \ldots, \cA_m) \\
& =\{R(a_1, \ldots, a_m) \ : \ (a_1, \ldots,  a_m)
\in \(\cA_1 \times \ldots \times \cA_m\)\setminus \cP_R\}, 
\end{align*}
where $\cP_R$ is the set of poles of $R$. 

We note that we have used $\cA^m$ for the 
$m$-fold Cartesian product   rather than for the  $m$-fold  product-set of a set $\cA$
as  the previous definition suggests. 
However neither of these notations is used in this section. 

For a scalar $\lambda \in \F_p$ we use the notation
$$
\lambda \cA = \{\lambda\} \cdot \cA  = \{\lambda a~:~a \in \cA\}, 
$$
for sets of multiples  of $\cA \subseteq \F_p$.

Applying the bound of Theorem~\ref{thm:T(G)} to  cosets of $\cG$, that is, to $T(\cG, \lambda \cG)$, 
and using~\eqref{eq:T and E} 
we obtain:

\begin{cor}
\label{cor:EnergyShiftSubgr}
Let $\cG$ be a multiplicative subgroup of $\F_p^*$. 
Then for any   $\lambda \in \F_p^* $, we have 
$$
\rE^\times (\cG + \lambda) -\frac{|\cG|^4}{p} \ll  \left\{
\begin{array}{ll}
p^{1/2} |\cG|^{3/2},& \text{if $ |\cG| \ge p^{2/3}$},\\
|\cG|^3 p^{-1/2} , & \text{if $p^{2/3} > |\cG| \ge p^{1/2}\log p$},\\
|\cG|^2 \log |\cG|, & \text{if $|\cG|< p^{1/2}\log p$}. 
\end{array}
\right.
$$
\end{cor}

Note that for $|\cG|/\sqrt{ p \log p}\to \infty$,  Corollary~\ref{cor:EnergyShiftSubgr} 
gives an asymptotic formula for $\rE^\times (\cG + \lambda)$; otherwise we only have an
upper bound.

\begin{cor}
\label{cor:AddRepr}For  a multiplicative subgroup $\cG$ of $\F_p^*$ and   $\lambda, \mu \in  \F_p^*$ 
we define the sets 
$$
\cS_1 = \cG +\lambda \cG + \mu \cG
\mand 
\cS_2 =
\left\{ \frac{u-   \lambda  }{v -\mu} ~:~  u,v \in 
\cG \right\}. 
$$
We have:
\begin{itemize}
\item if  $|\cG| \ge p^{2/3}$ then $\F_p^* \subseteq \cS_1$ and $\F_p^* \subseteq \cG \cS_2$;
\item if  $|\cG| \le (p \log p)^{1/2}$ then,  for $i=1,2$, 
$$
|\cS_i| \gg \frac{|\cG|^2}{\log |\cG|}\,;
$$
\item otherwise, for $i=1,2$, 
$$
p - |\cS_i| \ll  \left \{
\begin{array}{ll} p^{5/2}|\cG|^{-5/2},& \text{if $ |\cG| \ge p^{2/3}$},\\
 p^{3/2} |\cG|^{-1} , & \text{if $p^{2/3} >|\cG| \ge p^{1/2}\log p$},\\
p^{2} |\cG|^{-2}\log p, & \text{if $p^{1/2}\log p \ge |\cG|>  (p \log p)^{1/2}$}.
\end{array}
\right.
$$
\end{itemize}
\end{cor}

\begin{proof}
We consider the set $\cS_1$ first. 

First we show that $\cS_1 \supseteq \F^*_p$, provided $|\cS_1| \ge p^{2/3}$.
Clearly, the set $\cS_1$ satsifies the property $\cS_1 \cG = \cS_1$ and hence if $\cS_1 \supseteq \F^*_p$, then there is a nonzero $\xi$ such that $\cS_1 \cap \xi \cG = \emptyset$. 
In other words, the equation
$$
	x+\lambda y+z\mu =\xi w \,,\quad \quad x,y,z,w\in \cG
$$
has no solutions. By the orthogonality property of exponential functions, 
this means that  for the sum
$$
\sigma = 	\sum_{a \in \F_p} \sum_{x \in \cG} \ep (ax) \sum_{y \in \cG} \ep (a\lambda y)  
\sum_{z \in \cG} \ep (a\mu z)  \sum_{w \in \cG} \ep (-a\xi w)  
$$
we have
$$
\sigma  = 0.
$$
Clearly, the contribution of $\sigma$ corresponding to $a = 0$ equals $|\cG|^4$. 
Using the  well-known  bound 
$$
\left|\sum_{x\in \F_p} \ep(bx^{k}) \right | \le (k-1) p^{1/2}, \qquad b \in  \F_p^*, 
$$
see, for example~\cite[Equation~(1)]{HBK},  combined with the identity
$$
\sum_{z\in \cG} \ep(bz) = \frac{1}{k}\sum_{x\in \F_p^*} \ep(bx^{k}), 
$$
where $k = (p-1)/|\cG|$, 
we obtain
$$
	0 = \sigma \ge |\cG|^4 - p \sum_{a \in \F^*_p} \left|\sum_{x \in \cG} \ep (ax) \right| \left|\sum_{y \in \cG} \ep (a\lambda y) \right| \,.
$$
By the Cauchy inequality, we get
$$
	0 > |\cG|^4 - p^2 |\cG| \ge 0 
$$
and this is a contradiction which gives the result for  $|\cG| \ge p^{2/3}$.

We now consider  subgroups with $|\cG| < p^{2/3}$. 
Clearly
$$
|\cS_1 |  = | \cG +\lambda \cG + \mu \cG+\lambda\mu | \ge |(\cG+\lambda)(\cG+\mu)| . 
$$
For $a \in \F_p^*$, we let $N(a)$ be the number of solutions to $(g+\lambda)(h+\mu) = a$ with $g,h\in \cG$.
Clearly
$$
\sum_{a \in \F_p} N(a)= |\cG|^2.
$$
Hence, by the Cauchy inequality, we have
$$
|\cG|^4 =  \(\sum_{a \in \F_p} N(a)\)^2 \le  |(\cG+\lambda)(\cG+\mu)| \sum_{a \in \F_p} N(a)^2
= |(\cG+\lambda)(\cG+\mu)| F, 
$$
where $F$ is the number of solutions to 
$$
(g_1+\lambda)(h_1+\mu) = (g_2+\lambda)(h_2+\mu), \qquad g_1,g_2,h_1,h_2\in \cG.
$$
There are obviously $O(|\cG|^2)$ solutions when 
$$
(g_1+\lambda)(h_1+\mu) = (g_2+\lambda)(h_2+\mu)=0.
$$
For the other solutions we  repeat  the same argument as in the above. 
That is, for every  $a \in \F_p$, we first collect together solutions with the same value
$$
\frac{g_1+\lambda}{g_2+\lambda}=\frac{h_1+\mu}{h_2+\mu}=a.
$$ 
After this, using  the Cauchy inequality again, we obtain
$$
F \le \sqrt{ \rE^\times (\cG+\lambda)  \rE^\times (\cG+\mu)} + O(|\cG|^2).
$$
Hence, putting the above inequalities together, we derive, 
$$
|\cS_1 | \gg  \frac{|\cG|^4}{\sqrt{ \rE^\times (\cG+\lambda)  \rE^\times (\cG+\mu)}+ O(|\cG|^2)} . 
$$
Hence, using Corollary~\ref{cor:EnergyShiftSubgr},  we derive the result for $\cS_1$.
Indeed, let $\fR$ be the bound on $\left|\rE^\times (\cG + \lambda) -|\cG|^4/p\right|$ given by Corollary~\ref{cor:EnergyShiftSubgr}.
It is easy to see that for  $\cG$ to which  the upper bound on $p - |\cS_i|$ applies 
we have 
$$
\frac{|\cG|^4}{p} \gg \fR . 
$$
Hence 
\begin{align*}
 \rE^\times (\cG+\lambda)  \rE^\times (\cG+\mu) &
=\frac{|\cG|^8}{p^2} + O\(\frac{|\cG|^4}{p}\fR + \fR^2\) \\
&  = \frac{|\cG|^8}{p^2} + O\(\frac{|\cG|^4}{p}\fR\) \\
& = \frac{|\cG|^8}{p^2} \(1 + O\(\frac{p}{|\cG|^4}\fR\)\) 
\end{align*}
which, together with $\fR \gg |\cG|^2$,  implies 
\begin{align*}
\sqrt{ \rE^\times (\cG+\lambda)  \rE^\times (\cG+\mu)} 
+ O(|\cG|^2) &\\
  =  \frac{|\cG|^4}{p} \(1 + O\(\frac{p}{|\cG|^4}\fR\)\) & + O(|\cG|^2)  
  =  \frac{|\cG|^4}{p} \(1 + \fQ\),
\end{align*}
where 
$$
\fQ \ll \frac{p}{|\cG|^4}\fR. 
$$
 We note that by adjusting the implied constant in the upper bound on $p - |\cS_1|$ we see that one can actually assume that $|\cG| \ge C_0 (p \log p)^{1/2}$ for some sufficiently large absolute constant $C_0$, so that $|\fQ| \le 1/2$. 
 In this case 
 $$
  \(1 + \fQ\)^{-1} = 1 + O(\fQ) =  1+ O\(\frac{p}{|\cG|^4}\fR\)
 $$
 and the bound on $p - |\cS_1|$ follows. For the lower bound on $|S_1|$ we simply remark that the error term $\fR$ dominates the main term $|\cG|^4/p$ in 
 Corollary~\ref{cor:EnergyShiftSubgr}, so in this case we simply write 
 $$
\sqrt{ \rE^\times (\cG+\lambda)  \rE^\times (\cG+\mu)} 
+ O(|\cG|^2)  \ll \fR
$$
and the bound follows. 

Similar arguments also lead to the same bounds on $|\cS_2 |$. 
For example, consider the case  $|\cG| \ge p^{2/3}$ (where the statement about $\cS_2$ is 
slightly different than that about $\cS_1$).  We denote
\begin{equation} 
\label{eq: Set Q}
\cQ = \cG \cS_2=  \frac{\lambda \cG-\cG}{\mu \cG-\cG} .
\end{equation}
Using the orthogonality of exponential functions, for any $\xi  \in  \F_p^*$, 
we can write
\begin{align*}
\left|\left \{\lambda u_1-u_2 = \xi (\mu v_1-v_2)~:~u_i, v_i \in \cG, \ i =1,2\right\}\right | & \\
=
\frac{|\cG|^4}{p}  + \frac{1}{p}\sum_{a\in \F_p^*} \sum_{u \in \cG} \ep( a \lambda   u) 
  \sum_{v \in \cG} \ep(v) 
\sum_{w\in \cG}  \ep(a&\xi  \mu w) 
\sum_{z\in \cG} \ep(-a\xi z). 
 \end{align*}
As before, 
we obtain
\begin{align*}
\left|\left|\left \{\lambda u_1-u_2 = \xi(v_1-v_2)~:~u_i, v_i \in \cG, \ i =1,2\right\}\right | -  \frac{|\cG|^4}{p} \right |  &\\
< \sum_{a \in \F_p^*} \left| \sum_{w \in \cG} \ep(a \xi w)\right |^2 = p&|\cG| - |\cG|^2.
\end{align*}
Hence, for $|\cG| > p^{2/3}$ we have
\begin{align*}
\left|\left \{\lambda u_1-u_2 = \xi(v_1-v_2)~:~u_i, v_i \in \cG, \ i =1,2\right\}\right | &\\
>  \frac{|\cG|^4}{p} &-  p|\cG| + |\cG|^2 >  |\cG|^2.
\end{align*}
Therefore, there is a solution with $v_1\ne v_2$ which leads to a representation $\xi = (\lambda u_1-u_2)/(v_1-v_2)$
for every  $\xi\in  \F_p^*$.

Proofs of the other statements about    $\cS_2$ are the same as those about $\cS_1$.  
 \end{proof}

In particular, Corollary~\ref{cor:AddRepr} applies to $\cS_1 = \cG+\cG + \cG$ and  
$\cS_1 = \cG + \cG-\cG$.

We note that for the set $\cQ$ given by~\eqref{eq: Set Q} we have
$0\in \frac{\lambda \cG-\cG}{\mu \cG-\cG}$ if and only if $\lambda \in \cG$.

\begin{remark}
	Let 
	$$
	\cQ=\frac{ \lambda  \cG- \cG}{\cG-\cG} \mand 
	\cR=\frac{\lambda \cG-1}{\cG-1}. 
	$$
	Clearly, 
	$$
	\cR\cG = \cQ \mand \cR = 1-\cR. 
	$$ 
	Hence the set $\cQ$ contains both $\cR \cG$ and $(1-\cR)\cG$ and hence
	$|\cQ| \ge \max\{ |\cR \cG|, |(1-\cR)\cG| \}$. Using~\cite[Theorem~18]{Shkr5} 
	 and $|\cR| \gg |\cG|^2/ \log |\cG|$ one can show that there is an absolute constant $c > 0$ such that  $|\cQ| \gg |\cG|^{2+c}$ for sufficiently small $\cG$ (the condition $|\cQ|^2 |\cG| \le p^2$ is enough). 
	Thus the lower bound for size of $\cQ$ which follows from bounds on $|\cS_2|$ in  Corollary~\ref{cor:AddRepr} can be improved for small subgroups. 
\end{remark}

We note that Corollary~\ref{cor:AddRepr}  also allows us to obtain the following version of the {\it Romanoff
theorem\/} modulo almost all primes $p$.  

\begin{cor}
\label{cor:Rom} For a fixed integer $g$ with $|g|\ge 2$, and sufficiently 
large $Q$, for all but $o(Q/\log Q)$ primes $p \le Q$ every residue class 
modulo $p$ can be represented as $\ell + g^k + g^m + g^n$
for a prime $\ell < p$ and positive integers $k,m,n\le p-1$. 
\end{cor}

\begin{proof} We recall that by a special case of a  result of  Indlekofer and Timofeev~\cite[Corollary~6]{IndlTim}, 
given any positive $\alpha < 1$, for all but $o(Q/\log Q)$ primes $p \le Q$, the multiplicative 
order of $g$ modulo $p$ is at least $p^{1/2} \exp\((\log p)^{\alpha}\)$.
For each of these primes,  we apply  Corollary~\ref{cor:AddRepr} to the set $\cS_1 =  \cG+\cG + \cG$  
with the group  $\cG \equiv \langle g \pmod p \rangle$ 
(only the first two inequalities are relevant)  and use that  
for $ |\cG| \ge p^{2/3}$ we have 
$p^{5/2}|\cG|^{-5/2} \le  p^{3/2} |\cG|^{-1}$. 
Hence
we obtain 
$$
p-|\cS_1| \ll p^{3/2} |\cG|^{-1} \ll p \exp\(-(\log p)^{\alpha}\)  = o(p/\log p),
$$
and by the prime number theorem we conclude the proof.
\end{proof}  

We remark that a classical result of  Erd\H{o}s and Murty~\cite{ErdMur} can also be used
in the proof of Corollary~\ref{cor:Rom}, however the bound of~\cite[Corollary~6]{IndlTim} 
used in full strength allows to 
get better estimates on the size of the exceptional set.
Perhaps more recent results of Ford~\cite{Ford}  can also be used to 
estimate the size of the exceptional set, however we do not pursue this here.
\subsection{Possible application to arbitrary sets} 
%

Note that some auxiliary  results established in the proofs of~\cite[Theorems~1 and~2]{Glib} can be 
reformulated as bounds on the size of the set 
$(\cA-\cA)(\cA-\cA)$ for an arbitrary set $\cA \subseteq \F_p$. We also refer to~\cite{Bal} for more recent results and
references.  Combined with the ideas of Balog~\cite{Bal}  this may lead to further 
results on additive properties of the product sets of difference sets.

\section*{Acknowledgements}

The authors would like to thank Giorgis Petridis  for his 
comments and suggestions and  the referee for the very careful reading 
of the manuscript and numerous corrections. 

During the preparation of this work, 
the third author
was supported 
by the Australian Research Council Grant DP170100786.


\begin{thebibliography}{9999}

\bibitem{Aku} N. M. Akulinichev, `Estimates for rational 
trigonometric sums of a special type', 
{\it Doklady Acad. Sci. USSR\/}, {\bf 161} (1965), 743--745
(in Russian).

\bibitem{Bal} A. Balog,  
`Another sum-product estimate in finite fields', 
{\it Proc. Steklov Inst. Math.\/}, {\bf 280} (2013), Suppl.~2, S23--S29. 

\bibitem{Bourg1} J. Bourgain, 
`Multilinear exponential sums in prime
fields under optimal entropy condition on
the sources', {\it  Geom. and Funct. Anal.\/}, 
{\bf 18} (2009), 1477--1502.

\bibitem{Bourg2} J. Bourgain, `Estimates of polynomial 
exponential sums',   {\it Israel J. Math.\/}, {\bf 176} (2010), 
221--240.

%

\bibitem{BouGlib} J. Bourgain and A. Glibichuk, `Exponential sum estimates over a 
subgroup in an arbitrary finite field', {\it J. D'Analyse Math.\/}, {\bf 115} (2011), 51--70.


\bibitem{BGK} J. Bourgain,  A. A. Glibichuk and S. V. Konyagin,
`Estimates for the number of sums and products and for exponential
sums in fields of prime order', {\it J. Lond. Math. Soc.\/}, 
{\bf 73} (2006), 380--398.

\bibitem{BKT} J. Bourgain, N.H. Katz and T. Tao, 
`A sum-product estimate in finite fields, and applications', 
{\it Geom. Funct. Anal.\/}, {\bf 14} (2004), 27--57.



\bibitem{CoCoPi1} T. T. Cochrane, J. Coffelt and C. G. Pinner, 
`A further refinement of Mordell's bound on exponential sums', 
 {\it Acta Arith.\/}, {\bf 116}  (2005), 35--41.   

\bibitem{CoCoPi2} T. T. Cochrane, J. Coffelt and C. G. Pinner, 
`A system of simultaneous congruences arising from trinomial exponential sums', 
 {\it J. Theorie des Nombres, Bordeaux\/}, {\bf 18} (2006), 59--72.  

\bibitem{CoPi1} T. Cochrane and C. Pinner,
 `An improved Mordell type bound for exponential sums', 
 {\it Proc. Amer. Math. Soc.\/}, {\bf 133} (2005), 313--320.  
 
\bibitem{CoPi2} T. Cochrane and C. Pinner,
 `Using  Stepanov's method for exponential sums involving rational functions', 
 {\it J. Number Theory\/}, {\bf 116} (2006), 270--292.

\bibitem{CoPi3} T. Cochrane and C. Pinner,
 `Bounds on fewnomial exponential sums over $\Z_p$', 
 {\it Math. Proc. Camb. Phil. Soc.\/}, {\bf 149}  (2010), 217--227. 
 
\bibitem{CoPi4} T. Cochrane and C. Pinner,
 `Explicit bounds on monomial and binomial exponential sums', 
{\it Quart. J. Math.\/}, {\bf 62} (2011),   323--349. 

\bibitem{Croc} R. Crocker,  
`On the sum of a prime and of two powers of two', 
{\it Pacific J. Math.\/}, {\bf  36} (1971), 103--107. 
 
%
%

\bibitem{ErdMur} P. Erd{\H o}s and R. Murty, `On the order of $a \pmod p$',
{\it Proc. 5th Canadian Number Theory Association Conf.\/}, Amer.
Math. Soc.,  Providence, RI, 1999, 87--97.

\bibitem{Ford}
K. Ford,  `The distribution of integers with a divisor in a given interval',
{\it Ann. of Math.\/}, {\bf 168} (2008), 367--433.

\bibitem{Gar} M. Z. Garaev, `Sums and products of sets and estimates of rational
trigonometric sums in fields of prime order', 
{\it Russian Math. Surveys\/}, {\bf  65}  (2010),  599--658 
(Transl. from {\it Uspekhi Mat. Nauk\/}).
 
\bibitem{Glib} A.~Glibichuk,
`Combinational properties of sets of residues modulo a prime and the 
Erd{\H o}s-Graham problem', {\it  Math. Notes\/},
 {\bf 79}   (2006), 356--365 (Transl. from {\it Matem. Zametki\/}).
 
 \bibitem{HBK}
D. R. Heath-Brown and S. V. Konyagin, `New bounds for Gauss sums
derived from $k$th powers,
and for Heilbronn's exponential sum',
{\it Quart. J. Math.\/}, {\bf 51} (2000), 221--235.

\bibitem{IndlTim} H.-K. Indlekofer and N. M. Timofeev,
`Divisors of shifted primes', {\it Publ. Math. Debrecen\/},  {\bf
60} (2002), 307--345.


\bibitem{Kon} S. V. Konyagin, `Bounds of exponential sums over subgroups and
Gauss sums',  {\it Proc. 4th Intern. Conf. Modern Problems of Number Theory and Its
Applications\/}, Moscow Lomonosov State Univ., Moscow, 2002, 86--114 (in Russian).


\bibitem{Mac} S. Macourt, `Incidence results and bounds of trilinear and quadrilinear 
exponential sums', {\it Preprint\/}, 2017  (available from \url{http://arxiv.org/abs/1707.08268}).
    

\bibitem{Mit} D.~A.~Mit'kin, 
`Estimation of the total number of the rational points on a set of curves in a simple finite field', 
{\it Chebyshevsky Sbornik\/}, {\bf 4} (2003), no.4, 94--102 (in Russian).
   
\bibitem{MPR-NRS} B. Murphy, G. Petridis, O. Roche--Newton,  M. Rudnev and I.~D. Shkredov, `New results on sum--product type growth in positive characteristic', {\it Preprint\/}, 2017  (available from \url{http://arxiv.org/abs/1702.01003}).




%
 


\bibitem{PetShp}   G. Petridis and I. E. Shparlinski,
`Bounds on trilinear and quadrilinear exponential sums',
{\it  J. d'Analyse Math.\/},  (to appear).

\bibitem{Rom}
N. P. Romanoff, `\"Uber einige S\"atze der additiven Zahlentheorie', 
{\it  Math. Ann.\/}, {\bf 109} (1934), 668--678.

 
\bibitem{Shkr1}  I. D. Shkredov, 
`On exponential sums over multiplicative subgroups of medium
size', {\it Finite Fields and  Appl.\/}, {\bf 30} (2014), 
72--87.

\bibitem{Shkr2}  I. D. Shkredov, `On tripling constant of multiplicative subgroups',  
 {\it Integers\/}, {\bf 16} (2016),  \#A75. 
 

\bibitem{Shkr4}  I. D. Shkredov, `Differences of subgroups in subgroups', 
 {\it Integers\/}  (to appear). 
 

\bibitem{Shkr5}  I. D. Shkredov, `Some remarks on the asymmetric sum--product phenomenon', {\it Moscow J.   Combin.  and Number Theory\/} (to appear).

\bibitem{ShkVyu}  I. D. Shkredov and I. V. Vyugin,
`On additive shifts of multiplicative subgroups',
{\it Mat. Sb.\/}, {\bf 203} (2012), 81--100 (in Russian).


\bibitem{Shp1} 
I. E. Shparlinski, `On bounds of Gaussian sums', 
{\it Matem. Zametki\/}, 
{\bf 50} (1991), 122--130 (in Russian).

\bibitem{Shp2} I. E. Shparlinski, `On exponential sums with sparse
polynomials and rational functions',  {\it J. Number  Theory\/}, {\bf 60} (1996),
233--244.



\bibitem{Sht} 
Y. N. Shteinikov, `Estimates of trigonometric  sums over subgroups and some of their applications', 
{\it Matem. Zametki\/}, 
{\bf 98} (2015), 606--625 (in Russian).

\bibitem{StdeZe}
S. Stevens and F. de Zeeuw, 
`An improved point-line incidence bound over arbitrary fields',
{\it Bull. London Math. Soc.\/}  (to appear). 

\bibitem{Weil} A. Weil,
{\it Basic number theory\/}, Springer-Verlag, New York, 1974.


\bibitem{TaoVu}
T. Tao and V. Vu, {\it Additive combinatorics\/}, Cambridge Stud.
Adv. Math., {\bf 105}, Cambridge University Press, Cambridge, 2006.

\bibitem{Yu} H. B. Yu,  
`Estimates for complete exponential sums of special types, 
 {\it Math. Proc. Camb. Phil. Soc.\/}, {\bf 131} (2001),  321--326. 
 
 
 
%
%
%
%
%
%
%
%
%
%



%

%
%
%
%
%

 



\end{thebibliography}
\end{document}